\documentclass[]{amsart}
\usepackage[sorted-cites]{amsrefs}
\usepackage{amsfonts}
\usepackage{graphicx}
\usepackage{amscd}
\usepackage{amsmath}
\usepackage{amssymb}
\usepackage{xcolor}
\usepackage{pgfplots}
\pgfplotsset{compat=newest}

\makeatletter
\@namedef{subjclassname@2010}{%
  \textup{2010} Mathematics Subject Classification}
\makeatother

\usepackage{mathrsfs}
\usepackage{amsbsy}

\newtheorem{lemma}{\bf Lemma}

\newtheorem{proposition}{\bf Proposition}
\newtheorem{remark}{Remark}

\newtheorem{theorem}{\bf Theorem}

\def\TRic{\mathring{Ric}}

\def\RR{{\mathrm R}}
\def \2R{{\hat{\RR}}}

\def\det{\mathrm{det}}

\theoremstyle{definition}

\numberwithin{equation}{section}

\makeatletter
\@namedef{subjclassname@2020}{%
  \textup{2020} Mathematics Subject Classification}
\makeatother

\title[Four-dimensional Gradient Ricci Solitons]{Gradient shrinking Ricci Solitons and \\Modified Sectional Curvature}

\author{Xiaodong Cao}
\author{Ernani Ribeiro Jr}
\author{Hosea Wondo}

\address[X. Cao]{Department of Mathematics, Cornell University, Ithaca, NY 14853} \email{xiaodongcao@cornell.edu}

\address[E. Ribeiro Jr]{Departamento de Matem\'atica, Universidade Federal do Cear\'a - UFC, Campus do Pici, 60455-760, Fortaleza - CE, Brazil}
\email{ernani@mat.ufc.br}
\thanks{E. Ribeiro Jr was partially supported by CNPq/Brazil [305128/2025-6 and  351492/2025-9] and FUNCAP/Brazil [ITR-0214-00116.01.00/23].}

\address[H. Wondo]{Department of Mathematics, Cornell University, Ithaca, NY 14853} \email{h.wondo@outlook.com}


\date{October 1, 2025}

\begin{document}

\begin{abstract}
We investigate four-dimensional gradient shrinking Ricci solitons with positive modified sectional curvature. Our first main result shows that if the norm of the self-dual Weyl tensor and the scalar curvature satisfy a certain sharp pinching condition --- closely related, in a precise sense, to the K\"ahler case --- then the soliton is necessarily locally K\"ahler. We further obtain a characterization theorem and a weighted integral gap result for compact gradient shrinking Ricci solitons with modified sectional curvature bounded from below. In addition, we establish a Hitchin--Thorpe type inequality for compact four-dimensional Ricci solitons, providing new topological constraints on such ma\-ni\-folds.
\end{abstract}

\maketitle

\section{Introduction}
\label{int}

A smooth, complete Riemannian manifold $(M^n,g)$ is called a {\it gradient shrinking Ricci soliton} if there exists a smooth potential function $f$ on $M^n$ satisfying
\begin{equation}
\label{grs}
Ric+ Hess\,f= g,
\end{equation} 
where $Ric$ denotes the Ricci tensor of the metric $g$ and $Hess\,f$ is the Hessian of $f$. Gradient Ricci solitons play a fundamental role in the study of Hamilton’s Ricci flow \cite{Hamilton2}, as they arise naturally as singularity models for the Ricci flow \cites{Hamilton2,bamler2020structure,Topping,Sesum}. Their classification is a central problem in understanding the formation and structure of singularities in the Ricci flow \cites{Perelman2,bamler2020structure}. For further background, we refer the reader to the survey \cite{caoALM11} and the references therein.\\

 Hamilton \cite{Hamilton2} proved that any $2$-dimensional gradient shrinking Ricci soliton is either isometric to the Euclidean plane $\mathbb{R}^2$ or a quotient of the round sphere $\mathbb{S}^2$. In dimension $n=3$, results of Ivey \cite{Ivey}, Perelman \cite{Perelman2}, Naber \cite{Naber}, Ni--Wallach \cite{Ni}, and Cao--Chen--Zhu \cite{CaoA} show that every complete $3$-dimensional gradient shrinking Ricci soliton is isometric to a finite quotient of one of the following model spaces: the round sphere $\mathbb{S}^3$, the Gaussian shrinking soliton $\mathbb{R}^3$, or the round cylinder $\mathbb{S}^{2}\times\mathbb{R}$. In dimension $n=4$, a complete classification of $4$-dimensional gradient shrinking Ricci solitons remains open. For the special case of locally conformally flat solitons (i.e., $W=0$), results from \cites{ELM, Ni, zhang, PW2, CWZ, MS} establish that every complete four-dimensional gradient shrinking Ricci soliton is isometric to a finite quotient of $\mathbb{S}^4$, $\mathbb{R}^4$, or $\mathbb{S}^{3}\times\mathbb{R}$. Further classification results have been obtained under additional curvature assumptions, including:

 \begin{itemize}
     \item Half-conformally flat ($W^{+}=0$, for a suitable choice of orientation) by Chen and Wang \cite{CW};
     \item Bach-flat metrics by H.-D. Cao and Chen \cite{CaoChen};
     \item Harmonic Weyl tensor ($ \delta\, W=0$) by Fern\'{a}ndez-L\'opez and Garc\'ia-R\'io \cite{FLGR} and by Munteanu--Sesum \cite{MS}; 
     \item Harmonic self-dual Weyl tensor  ($\delta\, W^{+}=0$) by Wu et al. \cite{Wu} (see also \cite{CMM}).
 \end{itemize}
 
A recent result by Cheng and Zhou \cite{CZ2021}, together with work of Fern\'andez-L\'opez and Garc\'ia-R\'io \cite{Fl-Gr}, provides a complete classification of four-dimensional gradient shrinking Ricci solitons with constant scalar curvature. In the K\"ahler setting, a full classification has recently been obtained by Li and Wang \cite{LiWang}; see also \cites{BCCD22KahlerRicci, conlon2024classification,Tran2025, TZ1, TZ2}. For further contributions and related results, we refer to \cites{CaoChen, ELM, Ni, Chen, Chow, zhang, Naber, PW2, CWZ, FLGR, MW, KW, MS, MW2} and the references therein.\\

Although many classification results for four-dimensional gradient Ricci solitons rely on the vanishing of the Weyl tensor $W,$ it is natural to ask whether analogous conclusions can be obtained under weaker conditions, such as bounds on its norm or suitable curvature constraints. In this direction, Catino \cite{Catino} proved that any complete four-dimensional gradient shrinking Ricci soliton with nonnegative Ricci curvature satisfying
\begin{equation}
\label{cond01}
|W| R\leq \sqrt{3}\left(|\mathring{Ric}|-\frac{1}{2\sqrt{3}}S \right)^2
\end{equation} must be conformally flat. Here, $\mathring{Ric}$ denotes the traceless Ricci tensor and $S$ the scalar curvature. The assumption of nonnegative Ricci curvature was later removed by \cite{Wu}. In \cite{Zhang2}, Zhang showed that any four-dimensional gradient shrinking Ricci soliton with nonnegative and bounded Ricci curvature, and satisfying $$|W|\leq \eta \left| \mid \mathring{Ric}\mid -\frac{1}{2\sqrt{3}}S\right|,$$ for some constant $\eta<1+\sqrt{3}$, is either flat or has $2$-positive Ricci curvature. However, such classification results exclude the non-Einstein K\"ahler-Ricci soliton $\mathbb{S}^2\times \mathbb{R}^2$ endowed with the product metric (see \cite[Example 1]{CRZ}). To incorporate this example, H.-D. Cao, Ribeiro and Zhou \cite[Theorem 1]{CRZ} proved that a complete four-dimensional gradient shrinking Ricci soliton satisfying a suitable inequality controlling the norm of the self-dual Weyl tensor  $W^+$ is either Einstein, conformally flat, or isometric to $\mathbb{S}^{2}\times\mathbb{R}^{2}.$ Remarkably, no pointwise bound on the Ricci tensor is required. This result was later refined by
X. Cao, Ribeiro and Tran \cite{CRT} by assuming that
\begin{equation}
\label{cond11}|W^{+}|^2-\sqrt{6}|W^+|^3 \geq \frac{\sqrt{6}}{6}|\mathring{Ric}|^2|W^+|.
\end{equation} Their proof employs the notion of curvature of the second kind and certain algebraic identities developed by X. Cao, Gursky, and Tran \cite{CGT} in the resolution of a conjecture of Nishikawa.

In the first part of this article, we advance the classification program for four-dimensional gradient shrinking Ricci solitons by deriving new rigidity results under a curvature pinching condition involving the self-dual Weyl tensor $W^+$. The motivation for this approach comes from a classical identity of  Derdzi\'{n}ski \cite{derd1}, which asserts that any four-dimensional K\"ahler manifold, equipped with the natural orientation, satisfies the identity
$$|W^+|=\frac{S}{2\sqrt{6}}.$$ This exact relation between $W^+$ and $S$ reflects a deep interplay between complex geometry and conformal curvature decomposition in dimension four. Building on techniques developed in \cite{CRZ} and \cite{CRT}, we formulate a natural K\"ahler-type pinching condition involving $W^+$ and  $S$, which forces a gradient shrinking Ricci soliton to be locally a  K\"ahler-Ricci soliton, that is, K\"ahler after possibly pulling back to a double cover of $M^4$ (see \cite{LeBrun1}). This result uncovers a new bridge between curvature inequalities and complex geometric structure in the Ricci soliton setting.\\

We now state our first main theorem.

\begin{theorem}\label{ThmA}
    Let $(M^4,\,g,\,f)$ be a complete $4$-dimensional gradient shrinking Ricci soliton satisfying \eqref{grs}. If $(M^4,\,g,\,f)$ satisfies 
    \begin{equation}\label{Thm3A1_alt}
        \frac{S}{2\sqrt{6}}  \leq |W^+| \leq \frac{1}{\sqrt{6}}\left( 2-\frac{S}{2}\right),
    \end{equation}
    then $(M^4,\,g)$ is locally a K\"ahler-Ricci soliton.   
\end{theorem}

\begin{remark}
 We note that both inequalities in condition (\ref{Thm3A1_alt}) of Theorem \ref{ThmA} are attained as equalities for the K\"ahler-Ricci soliton $\mathbb{S}^2 \times \mathbb{R}^2$ equipped with the product metric.  
 \end{remark}

A key ingredient in the proof of Theorem~\ref{ThmA} is a Weitzenb\"ock-type formula for the self-dual Weyl tensor $W^{+}$ on four-dimensional gradient shrinking Ricci solitons, established by X. Cao and Tran in \cite{CH}. It asserts that any four-dimensional gradient shrinking Ricci soliton obeying (\ref{grs}) satisfies

\begin{equation}
\label{weitzenbock}
\Delta_{f} |W^{+}|^{2}=2|\nabla W^{+}|^{2}+4|W^{+}|^{2}-36\, \det\, W^{+}-\langle (\mathring{Ric}\odot \mathring{Ric})^{+},W^{+}\rangle,
\end{equation} where $\odot$ stands for the Kulkarni--Nomizu product. This formula (\ref{weitzenbock}) also plays a fundamental role in the proofs of Theorems~\ref{ThmB}, \ref{thmNew}, and~\ref{ThmC} in the present article, and features prominently in \cite{CRZ, CRT}.

Motivated by the work of Gursky and LeBrun \cite{LeBrun} and Berger \cite{Berger1}, X. Cao and Tran \cite{CH} introduced the notion of {\it modified sectional curvature} $\overline{K}$ on a four-dimensional gradient Ricci soliton $(M^4,\,g,\,f)$ via the modified curvature tensor 

\begin{equation}
\label{modR}
    \overline{R} := R + \frac{1}{2}Hess\,f\odot g.
\end{equation} Using this notion, they obtained integral estimates for the norm of the Weyl tensor on four-dimensional compact Ricci solitons. In the present work, we combine the concept of modified sectional curvature with bounds on the scalar curvature to derive new characterization results for compact gradient shrinking Ricci solitons. In this context, we establish the following theorem.

\begin{theorem}\label{ThmB}
Let $(M^4,\,g,\,f)$ be an oriented $4$-dimensional compact gradient shrin\-king Ricci soliton satisfying \eqref{grs}. If $$\overline{K} \geq \varepsilon := 1- \sqrt{\frac{268}{567}} \,(\approx 0.312)\,\,\,\hbox{ and }\,\,\,S \geq \delta := \frac{360}{67}(1-\varepsilon) = \frac{360}{67}\sqrt{\frac{268}{567}}\,\,(\approx 3.694),$$ then $M^4$ is isometric to either the standard sphere $\mathbb{S}^4$, or the complex projective space $\mathbb{C} \mathbb{P}^2.$ 
\end{theorem}

The proof of Theorem \ref{ThmB} is based on integral estimates obtained through the application of the Weitzenb\"ock-type formula (\ref{weitzenbock}), combined with key estimates from \cite{CH} and Catino’s gap theorem \cite{catinoAdv}. Subsequently, via a tour de force of integral estimates, combined with the classical gap theorem of Gursky and LeBrun \cite{LeBrun} and a rigidity theorem due to Yang \cite{yang2000rigidity}, we obtain the following result.

\begin{theorem}
\label{thmNew}
Let $(M^4,\,g,\,f)$ be an oriented $4$-dimensional compact gradient shrin\-king Ricci soliton satisfying \eqref{grs} and \begin{equation}
        \int_{M} |\delta W^{+}|^2\,dV_g \leq \int_{M} \frac{S}{6}|W^{+}|^2\, dV_g.
\end{equation} If $\overline{K} \geq0.3069 $ and $S \geq 3.668$,
then $M^4$ is isometric to either the standard sphere $\mathbb{S}^4$, or the complex projective space $\mathbb{C} \mathbb{P}^2.$ 
\end{theorem}

In the next part of the paper, motivated by the classical gap theorem of Gursky and LeBrun \cite{LeBrun}, we establish a weighted integral gap theorem for the half-Weyl tensor on four-dimensional compact gradient shrinking Ricci solitons.

\begin{theorem}
\label{ThmC}
    Let $(M^4,\,g,\,f)$ be a $4$-dimensional compact gradient shrinking Ricci soliton satisfying \eqref{grs}. If $\overline{K} \geq \varepsilon $, $S \geq \delta$ and $\frac{21}{2}<2 \delta+12 \varepsilon$, then the following inequality holds:
    \begin{equation}\label{4_III_Gap}
        \int_M|W^\pm|^2 e^{-f}\,dV_g \leq \alpha\int_M S^2 e^{-f}\,dV_g,
    \end{equation} where $\alpha =\frac{4(1-\varepsilon)}{\delta}-\frac{2}{3}-\left(1-\varepsilon\right)+ \frac{\delta}{6}$.
\end{theorem}

It is of significant interest to investigate the topological properties of compact four-dimensional gradient shrinking Ricci solitons. The classical Hitchin--Thorpe inequality  \cite{Thorpe, Hitchin} asserts that any compact four-dimensional Einstein manifold sa\-tis\-fies
\begin{equation}
\label{HTeq}
\chi(M)\geq 1.5\,|\tau (M)|,
\end{equation} where $\chi(M)$ and $\tau(M)$ denote the  Euler cha\-rac\-teristic and  the signature of $M^4,$ respectively (cf. \cite[Theorem 6.35]{Besse}). This inequality provides a purely topological obstruction to the existence of an Einstein metric on a given compact four-manifold.

Since gradient Ricci solitons are natural generalizations of Einstein ma\-ni\-folds-- and given that the known examples of nontrivial gradient shrinking Ricci solitons on $\mathbb{CP}^2\sharp (- \mathbb{CP}^2)$ and $\mathbb{CP}^2\sharp 2(- \mathbb{CP}^2)$ also satisfy (\ref{HTeq})--- it is natural to ask whether a similar topological restriction holds in this more general setting\footnote{As shown in \cite{Derd, FLGR0, Li}, any compact four-dimensional gradient shrinking Ricci soliton necessarily satisfies $\chi(M)>|\tau(M)|$ (i.e., Berger's inequality).}. This question was explicitly posed by H.-D. Cao \cite[Problem 6]{caoALM11}, and partial answers have been obtained in \cite{CRZpams,MR3128968,MR2672426,Tadano}. In the sequel, by imposing a lower bound on the modified sectional curvature $\overline{K},$ we establish a Hitchin--Thorpe type inequality for $4$-dimensional gradient shrinking Ricci solitons.

\begin{theorem}
\label{Thm_HT}
     Let $(M^4,\,g,\,f)$ be a $4$-dimensional oriented compact gradient shrin\-king Ricci soliton satisfying \eqref{grs} with $\overline{K} \geq \varepsilon = 0.184.$ If 
    \begin{equation}\label{6_Asump}
    \int_M |\delta W^{+}|^2\,dV_{g} \leq \int_M \frac{S}{6}|W^{+}|^2\, dV_{g},
    \end{equation} then  
    \begin{equation}
    \chi(M)>\frac{3}{2}\,|\tau(M)|.
    \end{equation}
\end{theorem}

\medskip

The remainder of this paper is organized as follows. In Section \ref{Sec2}, we review background material and key preliminary results that will be used in the proofs of our main theorems. In Section \ref{Sec3}, we prove Theorem \ref{ThmA}. Section \ref{Sec4} is divided into three parts: Section 4.1 contains the proofs of Theorem \ref{ThmB} and Theorem \ref{thmNew}; Section 4.2 is devoted to the proof of Theorem \ref{ThmC}; and Section 4.3 establishes Theorem \ref{Thm_HT}.

\section{Background}
\label{Sec2}

In this section, we review basic background material and collect several identities that will play a fundamental role in the proofs of our main results.

We begin by recalling some fundamental curvature tensors on a Riemannian manifold $(M^n,\,g)$ of dimension $n\ge 3$. The first is the {\it Weyl tensor} $W$, defined via the curvature decomposition formula:
\begin{eqnarray}
\label{weyl}
R_{ijkl}&=&W_{ijkl}+\frac{1}{n-2}\big(R_{ik}g_{jl}+R_{jl}g_{ik}-R_{il}g_{jk}-R_{jk}g_{il}\big) \nonumber\\
 &&-\frac{S}{(n-1)(n-2)}\big(g_{jl}g_{ik}-g_{il}g_{jk}\big),
\end{eqnarray} where $R_{ijkl}$ stands for the Riemann curvature tensor. It is well known that $W=0$ in dimension $n=3.$ The second one is the {\it Cotton tensor} $C$ given by
\begin{equation}
\label{cotton} \displaystyle{C_{ijk}=\nabla_{i}R_{jk}-\nabla_{j}R_{ik}-\frac{1}{2(n-1)}\big(\nabla_{i}R
g_{jk}-\nabla_{j}R g_{ik}).}
\end{equation} Notice that $C_{ijk}$ is skew-symmetric in the first two indices and trace-free in any two indices. Moreover, for $n\geq 4,$ we have 

\begin{equation}
C_{ijk}=-\frac{(n-2)}{(n-3)}\nabla_{l}W_{ijkl}.
\end{equation} 

From now on, we restrict our attention to the case $n=4.$ On an oriented Riemannian manifold $(M^4,\,g),$ the bundle of $2$-forms $\Lambda^2$ can be  decomposed into a direct sum
\begin{equation}
\label{lk1}
\Lambda^2=\Lambda^{+}\oplus\Lambda^{-}.
\end{equation} This decomposition is conformally invariant and allows us to decompose the Weyl tensor, viewed as an endomorphism of $\Lambda^2$, into components  
\begin{equation}
\label{ewq}
W = W^+\oplus W^-,
\end{equation} where $W^\pm:\Lambda^\pm M\longrightarrow\Lambda^\pm M$ are called the {\it self-dual} part and {\it anti-self-dual} part of the Weyl tensor, respectively. 

Fixing a point $p\in M^4$, we diagonalize $W^\pm$ so that their respective eigenvalues are denoted by $\lambda_i$ and $\mu_{i},$ $1\le i \le 3.$ In particular, we have
\begin{equation}
\label{eigenvalues}
\begin{cases}
 \lambda_1\geq \lambda_2 \geq \lambda_3\, \,\,\,  \hbox{and}\,\,\,\, \lambda_1 +\lambda_2 +\lambda_3= 0, &\\
\mu_1\geq \mu_2 \geq \mu_3\,\,\,\,  \hbox{and}\,\,\,\,\mu_1 +\mu_2 +\mu_3= 0.&
\end{cases}
\end{equation}

Next, we recall the following Kato inequality. 

\begin{lemma}
	\label{Katoinequality}
	Let $(M^{4},\,g)$ be an oriented four-dimensional Riemannian ma\-ni\-fold. Then we have:
	\[|\nabla |W^+|| \leq |\nabla W^+|.\]
	Equality holds if and only if there is a one-form $\nu$ such that
	$\nabla W^+=\nu \otimes W^+.$ 
\end{lemma}

We will also use the following algebraic inequality

\begin{equation}
	\label{eqdet}
	\det\, W^{+}\leq \frac{\sqrt{6}}{18}|W^{+}|^{3}.
\end{equation} Moreover, equality holds in (\ref{eqdet}) if and only if $\lambda_{2}=\lambda_{3}=-\frac{1}{2}\lambda_{1}.$ 

For completeness, we present here the proof of \eqref{eqdet}. From \eqref{eigenvalues}, we have

\begin{equation*} |W^{+}|^2 = (\lambda_{1})^2 + (\lambda_{2})^2 + (\lambda_{3})^2\geq (\lambda_{1})^2 +\frac{(\lambda_{2}+\lambda_{3})^2}{2}=\frac{3}{2}(\lambda_{1})^2. \end{equation*} Consequently, \begin{equation*} \det\, W^{+} = \lambda_{1}\lambda_{2}\lambda_{3}\leq \lambda_{1}\frac{(\lambda_{2}+\lambda_{3})^2}{4}=\frac{1}{4}(\lambda_{1})^3 \leq \frac{\sqrt{6}}{18}|W^{+}|^{3}, \end{equation*} as claimed.

The curvature and topology of a compact $4$-dimensional manifold are related through the classical Gauss--Bonnet--Chern formula
\begin{equation}
\label{6_HTest3}
\chi(M)=\frac{1}{8\pi^2}\int_M\left(|W^+|^2+|W^-|^2+\frac{S^2}{24}-\frac{1}{2}|\mathring{Ric}|^2\right)dV_g,
\end{equation} and Hirzebruch's theorem 
\begin{equation}
\label{6_HTest311}
\tau(M)=\frac{1}{12\pi^{2}}\int_M\left(|W^+|^2-|W^-|^2\right)dV_{g},
\end{equation} where $\chi(M)$ and $\tau(M)$ denote the  Euler cha\-rac\-teristic and  the signature of $M^4,$ respectively; see \cite[Chapter 13]{Besse} for more details. It follows from (\ref{6_HTest3}) and (\ref{6_HTest311}) that every compact $4$-dimensional Einstein manifold must satisfy the Hitchin--Thorpe inequality \cite{Thorpe,Hitchin}, namely, 
\begin{equation}
\label{HTeq11}
\chi(M)\geq 1.5\,|\tau (M)|.
\end{equation} See also \cite[Theorem 6.35]{Besse}.

\subsection{Four-dimensional gradient Ricci solitons}
From now on, let ($M^4,\,g,\, f$) be a complete $4$-dimensional gradient shrinking Ricci soliton satisfying 

\begin{equation}
\label{eqGRS} 
Ric+Hess\,f=g. 
\end{equation} Tracing the soliton equation (\ref{eqGRS}) we get 

\begin{equation}
\label{traceGRS}
S+\Delta f=4,
\end{equation} where $S$ denotes the scalar curvature of $M^4.$

We now collect some well-known identities for gradient shrinking Ricci solitons.

\begin{lemma}[\cite{Hamilton2}]
\label{GRSiden}
Let ($M^4,\,g,\, f$) be a $4$-dimensional gradient shrinking Ricci soliton satisfying \eqref{grs}. Then we have:
\begin{enumerate}
    \item $\nabla S=2 Ric(\nabla f)$;
    \vspace{0.20cm}
    \item $\Delta_f S=2  S-2|\text{Ric}|^2 $;
        \vspace{0.20cm}
    \item $S+|\nabla f|^2=2f,$ after normalizing;
        \vspace{0.20cm}
    \item $\Delta_f R_{ij}=2R_{i j}-2 R_{i k j l} R_{k l}.$
\end{enumerate}
Here, $\Delta_f:=\Delta \cdot-\nabla_{\nabla f} \cdot$ denotes the drifted Laplacian.
\end{lemma} In particular, one sees from assertion (2) of Lemma \ref{GRSiden} that
\begin{equation}\label{trRic_RiclapS}
\Delta_f S = 2S - \frac{S^2}{2}-2|\mathring{Ric}|^2,
\end{equation} where 
\begin{equation}\label{Tricdef}
    |\mathring{Ric}|^2=|Ric|^2-\frac{S^2}{4}.
\end{equation}

In \cite{Chen}, Chen showed that any complete ancient solution to the Ricci flow has nonnegative scalar curvature, which implies that $S\geq 0$ for any complete gradient shrinking Ricci soliton. Moreover, $S$ is strictly positive  unless $(M^4,\,g,\,f)$ is the Gaussian shrinking soliton (see \cite{PRS}). Regarding the potential function $f$, H.-D. Cao and Zhou \cite{CZ} proved that 
\begin{equation}
\label{eqfbeh}
\frac{1}{4}\Big(r(x)-c\Big)^{2}\leq f(x)\leq \frac{1}{4}\Big(r(x)+c\Big)^{2},
\end{equation} for all $r(x)\geq r_{0},$ where $r=r(x)$ is the distance function to a fixed point in $M.$ Additionally, they showed that every complete noncompact gradient shrinking Ricci soliton has at most Euclidean volume growth (see \cite[Theorem 1.2]{CZ}). Moreover, these asymptotic estimates are optimal in the sense that they are achieved by the Gaussian shrinking soliton.

As previously mentioned, X. Cao and Tran \cite{CH} showed that any four-dimensional gradient shrinking Ricci soliton obeying  (\ref{grs}) satisfies the following Weitzenb\"ock-type formula. 

\begin{proposition}[\cite{CH}]
\label{propWeitz}
Let ($M^4,\,g,\, f$) be a $4$-dimensional gradient shrinking Ricci soliton satisfying \eqref{grs}. Then we have:
\begin{equation}
\label{weitzenbock_B}
\Delta_{f} |W^{+}|^{2}=2|\nabla W^{+}|^{2}+4|W^{+}|^{2}-36\, \det W^{+}-\langle (\mathring{Ric}\odot \mathring{Ric})^{+},W^{+}\rangle,
\end{equation} where $\odot$ stands for the Kulkarni--Nomizu product. 
\end{proposition}

The Kulkarni-Nomizu product $\odot,$  which takes two symmetric $(0, 2)$-tensors and yields a $(0, 4)$-tensor with the same algebraic symmetries as the curvature tensor, is defined by
\begin{eqnarray}
\label{eq76}
(A \odot B)_{ijkl}= A_{ik}B_{jl}+A_{jl}B_{ik}-A_{il}B_{jk}-A_{jk}B_{il}.
\end{eqnarray}

It is also important to recall a sharp estimate for the term involving the traceless Ricci tensor and the self-dual Weyl tensor obtained in \cite[Lemma 6]{CRT}.

\begin{lemma}
	\label{lemK}
	Let $(M^{4},\,g)$ be an oriented $4$-dimensional Riemannian manifold. Then we have:     	
	\begin{equation}
		\label{EqHu}
		\langle (\mathring{Ric}\odot \mathring{Ric})^{+},W^{+}\rangle \leq \frac{\sqrt{6}}{3}|\mathring{Ric}|^{2}|W^{+}|.
	\end{equation} Moreover, equality holds if and only if $W^+$ has eigenvalues 
	$$0\leq \lambda_1=-2 \lambda_2=-2\lambda_3.$$
\end{lemma} 

As a consequence of the equality case in Lemma \ref{lemK} and (\ref{eqdet}), we have the following proposition.

\begin{proposition}[\cite{CRT}]
	\label{kahlerform}
		Let $(M^4,\ g)$ be a $4$-dimensional Riemannian manifold. If $\nabla W^+= \nu \otimes W^+$ for some one-form $\nu$ and $\lambda_2=\lambda_3$ at each point, then $\omega_1$ is a locally K\"ahler form.
\end{proposition}

We now recall the concept of {\it modified curvature tensor} $\overline{R},$ introduced by X. Cao and Tran \cite{CH}, for a four-dimensional gradient Ricci soliton $(M^4,\,g,\,f).$ It is defined in \eqref{modR} by 
\begin{equation}
\label{modRaa}
    \overline{R} = R + \frac{1}{2}Hess\,f\odot g.
\end{equation} From this, it follows that

\[
\overline{R} = \begin{pmatrix}
\bar{A}^{+} & O \\
O & \bar{A}^{-}
\end{pmatrix},
\] with $\bar{A}^{\pm}=W^{\pm}+ (1-\frac{S}{6})Id.$ In particular, as observed in \cite[Proposition 2.4]{CH} (see also \cite{Berger1}), the modified curvature tensor $\overline{R}$ admits a normal form given by

\[
\overline{R} = \begin{pmatrix}
A & B \\
B & A
\end{pmatrix},
\] with $A=diag(a_{1},a_{2},a_{3})$ and $B=diag(b_{1},b_{2},b_{3}).$ Moreover, $a_1=\min \overline{K},$ $a_{3}=\max \overline{K}$ and $|b_{i}-b_{j}|\leq |a_{i}-a_{j}|,$ where $\overline{K}$ is the modified sectional curvature, that is, $\overline{K}(e_{1},e_{2})=\overline{R}_{1212}$ for any orthonormal vector $e_{1}$ and $e_{2}.$

By assuming a lower bound on the modified sectional curvature $\overline{K},$  X. Cao and Tran \cite[Lemma 2.5]{CH} derived the following essential result.

\begin{proposition}[\cite{CH}]
Let ($M^4,\,g,\, f$) be a $4$-dimensional gradient shrinking Ricci soliton satisfying \eqref{grs} with $\overline{K} \geq \varepsilon.$ Then the following assertions hold:

\begin{equation}\label{Kbarboundineq}
\begin{aligned}
S+3 \Delta f & \geq 12 \varepsilon , \\
S & \leq 6(1-\varepsilon) , \\
\Delta f & \geq 2(3 \varepsilon-1) , \\
\frac{1}{\sqrt{6}}\left(\left|W^{+}\right|+\left|W^{-}\right|\right) & \leq 2(1-\varepsilon) -\frac{S}{3} .
\end{aligned}
\end{equation} 
\end{proposition}

Finally, we need to recall two fundamental gap theorems that will play a crucial role in the proofs of Theorems \ref{ThmB} and \ref{thmNew}. The first is the classical gap theorem for Einstein manifolds, established by Gursky and LeBrun \cite{LeBrun}.

\begin{theorem}[\cite{LeBrun}]
\label{Gursky_Gap}
Let $(M^4,\, g)$ be a compact oriented $4$-dimensional Einstein manifold with positive scalar curvature $S$ and $W^{+} \not \equiv 0.$ Then we have:

$$ \int_M\left|W^{+}\right|^2\, dV_g \geq \int_M \frac{S^2}{24}\, dV_g,
$$ with equality if and only if $\nabla W^{+} \equiv 0$.
\end{theorem} 

We note that equality in Theorem~\ref{Gursky_Gap} is achieved both by $\mathbb{CP}^2,$ equipped with the Fubini--Study metric and its standard orientation, and $\mathbb{S}^2 \times \mathbb{S}^2$ with the product metric.

The second gap theorem was obtained by Catino \cite{catinoAdv} in the context of compact gradient shrinking Ricci solitons. This result can be compared to the sphere theorem of Chang, Gursky, and Yang \cite[Theorem A']{ChangGurskyYang}, which provides a rigidity characterization under a pinching condition involving the Weyl tensor and scalar curvature on four-dimensional compact manifolds.

\begin{theorem}[\cite{catinoAdv}]
\label{Catino_Pinching}
Any $4$-dimensional compact gradient shrinking Ricci soliton ($M^4,\,g,\, f)$ satisfying the integral pinching condition

$$
\int_M|W|^2 d V_g+\frac{5}{4} \int_M|\overset{\circ}{Ric}|^2 d V_g \leq \frac{1}{48} \int_M S^2 d V_g
$$ is isometric to a quotient of the round $\mathbb{S}^4$.
\end{theorem}

\section{Weyl-scalar curvature pinching}
\label{Sec3}

In this section, we present the proof of Theorem \ref{ThmA}. To that end, we first establish a key preliminary result.

\begin{proposition}
    Let $(M^4,\,g,\,f)$ be a gradient shrinking Ricci soliton satisfying \eqref{grs}. Then we have: 
    \begin{equation}
    \label{3_Weizenbock_noncompact}
    |W^+|\Delta_f\left(|W^+| - \frac{S}{2 \sqrt{6}} \right) \geq  |W^+| \left(|W^+|- \frac{S}{2 \sqrt{6}} \right)\left(2 - \sqrt{6}|W^+|- \frac{S}{2 }\right) .
\end{equation}
\end{proposition}
\begin{proof}

To begin, we invoke Proposition \ref{propWeitz}:

\begin{equation}
\label{weitzenbock_B_1}
\Delta_{f} |W^{+}|^{2}=2|\nabla W^{+}|^{2}+4|W^{+}|^{2}-36\, \det W^{+}-\langle (\mathring{Ric}\odot \mathring{Ric})^{+},W^{+}\rangle,
\end{equation} together with (\ref{eqdet}) and Lemma \ref{lemK} to infer 
\begin{equation}\label{3_lapW+}
    \frac{1}{2}\Delta_f\left|W^{+}\right|^2 \geq  \left|\nabla W^{+}\right|^2 +2\left|W^{+}\right|^2-\sqrt{6}| W^{+}|^3-\frac{1}{\sqrt{6}}|\overset{\circ}{Ric}|^2|W^+|.
\end{equation} This, combined with \eqref{trRic_RiclapS}, yields
\begin{eqnarray}\label{3_2PWineq}
    \frac{1}{2}\Delta_f\left|W^{+}\right|^2 &\geq &  \left|\nabla W^{+}\right|^2 +2\left|W^{+}\right|^2-\sqrt{6}| W^{+}|^3-\frac{S}{\sqrt{6}}|W^+| \nonumber\\&&+ \frac{S^2}{4 \sqrt{6}}|W^+| + \frac{1}{2 \sqrt{6}}(\Delta_f S)|W^+|.
\end{eqnarray} By Kato's inequality 
\begin{eqnarray*}
    \left| W^{+}\right| \Delta_f\left| W^{+}\right| &\geq & 2\left| W^{+}\right|^2-\sqrt{6}| W^{+}|^3-\frac{S}{\sqrt{6}}|W^+| \nonumber\\&&+ \frac{S^2}{4 \sqrt{6}}|W^+| + \frac{1}{2 \sqrt{6}}(\Delta_f S)|W^+|,
\end{eqnarray*} where we have used that the left-hand side of \eqref{3_2PWineq} is 
\begin{equation*}
\frac{1}{2}\Delta_f\left|W^{+}\right|^2=\left|W^{+}\right| \Delta_f\left|W^{+}\right|+|\nabla| W^{+}| |^2.
\end{equation*} Rearranging terms, one obtains that
\begin{equation*}
\begin{aligned}
    |W^+|\Delta_f\left(|W^+| - \frac{S}{2 \sqrt{6}} \right) &\geq  2|W^+| \left( |W^+| - \frac{S}{2 \sqrt{6}}\right) - \sqrt{6}|W^+| \left(|W^+|^2-\frac{S^2}{24}\right)\\
    &= |W^+| \left(|W^+|- \frac{S}{2 \sqrt{6}} \right)\left(2 - \sqrt{6}|W^+|- \frac{S}{2}\right),
\end{aligned}
\end{equation*} as asserted. 
\end{proof}

We are now ready to present the proof of Theorem \ref{ThmA}.

\begin{proof}[{\bf Proof of Theorem \ref{ThmA}}]

We first set $$\Phi:= |W^+|-\frac{S}{2\sqrt{6}}.$$
Note that our assumption guarantees that the left-hand side of \eqref{3_Weizenbock_noncompact} is non-negative and hence, $\Delta_f \Phi$ is also non-negative. For the compact case, it suffices to use the maximum principle to conclude that \begin{equation}\label{5_WSequal}
    |W^+| = \frac{S}{2\sqrt{6}}+C,
\end{equation} where $C$ is a constant.

On the other hand, if $M^4$ is noncompact, we consider a cut-off function $\rho: M \rightarrow \mathbb{R}$ such that $\rho=1$ on a geodesic ball $B_p(r)$ centered at a fixed point $p \in M$ of radius $r, \rho=0$ outside of $B_p(2 r)$ and $|\nabla \rho| \leq \frac{c}{r}$, where $c$ is a constant. Integrating by parts, one sees that
$$
\begin{aligned}
    0  &\geq - \int_M\rho^2 \Phi (\Delta_f\Phi) e^{-f}dV_g \\
    & =\int_M\left\langle\nabla\left(\rho^2\Phi\right), \nabla \Phi \right\rangle e^{-f} d V_g\\
    & = \int_M\left|\nabla\left(\rho \Phi\right)\right|^2 e^{-f} d V_g - \int_M \Phi^2|\nabla \rho|^2 e^{-f} d V_g,
\end{aligned}
$$
so that,
$$
    \int_M\left|\nabla\left(\rho \Phi\right)\right|^2 e^{-f} d V_g \leq \int_M \Phi^2|\nabla \rho|^2 e^{-f} d V_g.
$$ From this, it follows that

\begin{eqnarray}
\label{5_radial_est}
    \int_{B(r)}\left|\nabla \Phi\right|^2 e^{-f} d V_g &\leq & \int_M\left|\nabla\left(\rho\Phi\right)\right|^2 e^{-f} d V_g \nonumber\\
    &\leq & \int_M\Phi^2|\nabla \rho|^2 e^{-f} d V_g \nonumber \\
    &\leq & \int_{M \backslash B(2 r)}\Phi^2|\nabla \rho|^2 e^{-f} d V_g+\int_{B(2 r) \backslash B(r)}\Phi^2|\nabla \rho|^2 e^{-f} d V_g\nonumber\\&& +\int_{B(r)}\Phi^2|\nabla \rho|^2 e^{-f} d V_g\nonumber \\
    &\leq & \int_{B(2 r) \backslash B(r)}\Phi^2|\nabla \rho|^2 e^{-f} d V_g \nonumber \\
    &\leq & \frac{c^2}{r^2} \int_M\Phi^2 e^{-f} d V_g.
\end{eqnarray}

We now need to show that $\Phi$ is $L^2_f$-integrable. Indeed, using the upper bound in \eqref{Thm3A1_alt}, one obtains that

$$
\begin{aligned}
    \Phi^2 = \left(|W^+|-\frac{S}{2\sqrt{6}} \right)^2 & \leq \frac{1}{6}\left( 2 - \frac{S}{2}\right)^2  + \frac{S^2}{24} \\
    & \leq \frac{2}{3} + \frac{S^2}{12} \\
    & \leq \frac{2}{3} + \frac{1}{3} |Ric|^2,
\end{aligned}
$$ where we have used that $S\geq 0.$ Consequently,

\begin{eqnarray}
    \int_{M} \Phi^2 e^{-f}dV_g\leq \frac{2}{3}\int_{M}e^{-f}dV_g + \frac{1}{3}\int_{M}|Ric|^2 e^{-f}dV_g.
\end{eqnarray} Since $M^4$ has finite weighted volume \cite[Corollary 1.1]{CZ} and $|\mathring{Ric}|$ is $L_f^2$-integrable \cite{MS}, one sees that $\Phi$ is also $L_f^2$-integrable, namely, 
$$\int_M \Phi^2 e^{-f} d V_g <\infty.
$$

Returning to (\ref{5_radial_est}), by taking limit as $r \rightarrow \infty,$ one concludes that $\Phi$ is constant, i.e., (\ref{5_WSequal}) holds.

Finally, in both the compact and noncompact cases, substituting equation (\ref{5_WSequal}) into Proposition \ref{propWeitz} yields the equality cases in (\ref{eqdet}) and Lemma \ref{lemK}. Therefore, we apply Proposition \ref{kahlerform} to conclude that $M^4$ is locally a K\"ahler-Ricci soliton. So, the proof is completed. 
\end{proof}

 \section{Lower bounds on the modified sectional curvature}
 \label{Sec4}
 
In this section, we present the proofs of Theorems \ref{ThmB}, \ref{thmNew}, \ref{ThmC} and \ref{Thm_HT}.

\subsection{Characterization under lower bounds}
We begin by proving Theorem~\ref{ThmB}. For the reader’s convenience, we restate it here.

\begin{theorem}[Theorem \ref{ThmB}]
\label{ThmB_1}
Let $(M^4,\,g,\,f)$ be an oriented $4$-dimensional compact gradient shrinking Ricci soliton satisfying \eqref{grs}. Suppose that 
$$\overline{K} \geq \varepsilon := 1- \sqrt{\frac{268}{567}} \,(\approx 0.312)\,\,\hbox{and}\,\,S\geq \delta := \frac{360}{67}(1-\varepsilon) = \frac{360}{67}\sqrt{\frac{268}{567}}\, (\approx 3.694).$$ Then $M^4$ is isometric to either the standard sphere $\mathbb{S}^4$, or the complex projective space $\mathbb{C} \mathbb{P}^2.$ 
\end{theorem}
\begin{proof} Initially, we combine Proposition \ref{propWeitz},  (\ref{eqdet}) and Lemma \ref{lemK} in order to obtain  
\begin{eqnarray}
\label{3_WIdentity_g}
       \frac{1}{2}\Delta |W^+|^2 &\geq& \frac{1}{2}\nabla_{\nabla f} |W^+|^2+|\nabla W^+|^2 +2 |W^+|^2-\sqrt{6}|W^+|^3\nonumber\\&& - \frac{1}{\sqrt{6}}|\mathring{Ric}|^2|W^+|.
\end{eqnarray} Hence, on integrating by parts, one sees that
\begin{eqnarray*}
 0 & \geq & \int_M \Big(-\frac{1}{2}(\Delta f)|W^+|^2+|\nabla W^+|^2+2|W^+|^2-\sqrt{6}|W^+|^3 \nonumber\\&&- \frac{1}{\sqrt{6}}|\mathring{Ric}|^2|W^+|\Big)dV_g.
 \end{eqnarray*} By using (\ref{traceGRS}), we have
\begin{equation}\label{4_Weizenbock_unormalised_int}
    0 \geq \int_M \Big( |\nabla W^+|^2+ \frac{S}{2}|W^+|^2 - \sqrt{6}|W^+|^3 - \frac{1}{\sqrt{6}}|\mathring{Ric}|^2|W^+| \Big)dV_g.
\end{equation}

Similarly, (\ref{4_Weizenbock_unormalised_int}) holds for $|W^-|.$  Therefore, by combining these two inequalities, we infer
$$
\begin{aligned}
        0 \geq \int_M &\Big[|\nabla W^+|^2 + |\nabla W^-|^2+ \frac{S}{2}\left(|W^+|^2+|W^-|^2\right) \\
        & - \sqrt{6}(|W^+|+|W^-|)(|W^+|^2-|W^+||W^-| + |W^-|^2) \\
        & - \frac{1}{\sqrt{6}}|\TRic|^2\left( |W^+|+|W^-|\right)\Big]dV_g.
\end{aligned}
$$ Now, by using \eqref{Kbarboundineq} into the fourth term above (noting that the second factor is always nonnegative), one deduces that

\begin{eqnarray}
\label{ineq871p}
        0 &\geq & \int_M \Big[|\nabla W^+|^2 + |\nabla W^-|^2+ \frac{S}{2}\left(|W^+|^2+|W^-|^2\right)\nonumber \\
        && - 12(1-\varepsilon)(|W^+|^2 + |W^-|^2) + 12(1-\varepsilon)|W^+||W^-| \nonumber\\
        && +2S(|W^+|^2+|W^-|^2) - 2S|W^+|W^-|\nonumber \\
        && - 2(1-\varepsilon)|\mathring{Ric}|^2+\frac{S}{3} |\mathring{Ric}|^2 \Big] dV_g.
\end{eqnarray}

On the other hand, on integrating (\ref{trRic_RiclapS}) and using (\ref{traceGRS}), one obtains that

\begin{eqnarray}
\label{intricball}
2\int_{M} |\mathring{Ric}|^2 dV_{g}&=& \int_{M} \left(2S -\frac{S^2}{2}-\Delta_{f}S\right) dV_g\nonumber\\ &=& 8 Vol(M)-\frac{1}{2}\int_{M}S^2 dV_{g} +\int_{M}\langle \nabla f,\,\nabla S\rangle dV_{g}\nonumber\\&=& 8 Vol(M)-\frac{1}{2}\int_{M}S^2 dV_{g} +\int_{M}S^{2} dV_{g}\nonumber\\&&-16 Vol(M)\nonumber\\&=& \frac{1}{2}\int_{M}S^2 dV_g - 8 Vol(M).
\end{eqnarray} Plugging this into (\ref{ineq871p}), we get

\begin{eqnarray*}
        0 &\geq & \int_M \Big[ |\nabla W^+|^2 + |\nabla W^-|^2 +\Big( \frac{5}{2}S -12(1-\varepsilon)-\frac{4}{15}\delta\Big) |W|^2 \nonumber\\
        && +\frac{4}{15}\delta |W|^2 + \frac{S}{3}|\mathring{Ric}|^2 - \Big(\frac{1-\varepsilon}{2}\Big)S^2 \nonumber\\
        && +(12(1-\varepsilon)-2S)|W^+||W^-| + 8(1-\varepsilon)\Big]\,dV_g.
\end{eqnarray*} By the second inequality in \eqref{Kbarboundineq}, one sees that $$(12(1-\varepsilon)-2S)|W^+||W^-|\geq 0,$$ which implies 

\begin{eqnarray}
\label{lmn10i}
        0 &\geq & \int_M \Big[ |\nabla W^+|^2 + |\nabla W^-|^2 +\Big( \frac{5}{2}S -12(1-\varepsilon)-\frac{4}{15}\delta\Big) |W|^2 \nonumber\\
        && +\frac{4}{15}\delta |W|^2 + \frac{S}{3}|\mathring{Ric}|^2 - \Big(\frac{1-\varepsilon}{2}\Big)S^2  + 8(1-\varepsilon)\Big]\,dV_g.
\end{eqnarray} Next, it follows from Theorem \ref{Catino_Pinching} (cf. \cite{catinoAdv}) that
$$
\int_M |W|^2 dV_g + \frac{5}{4}\int_M |\mathring{Ric}|^2dV_g \geq \frac{1}{48} \int_M S^2 dV_g.
$$ This substituted into the fourth term in the right-hand side of (\ref{lmn10i}) yields

\begin{equation}\label{4_IV_final}
    \begin{aligned}
        0 \geq \int_M & \Big[|\nabla W^+|^2 + |\nabla W^-|^2 +\left( \frac{5S}{2} -12(1-\varepsilon)-\frac{4}{15}\delta\right) |W|^2 \\
        & +\left(\frac{S-\delta}{3}\right)|\mathring{Ric}|^2 + \left(\frac{\delta}{180}-\frac{1-\varepsilon}{2}\right)S^2  + 8(1-\varepsilon)\Big]\,dV_g.
\end{aligned}
\end{equation}
Our assumption $S \geq \delta,$ together with the chosen values of $\varepsilon$ and $\delta,$ guarantee that the third and fourth terms on the right-hand side of the inequality (\ref{4_IV_final}) are nonnegative. Moreover, observe that the coefficient of $S^2$ simplifies to $-\frac{63}{134}(1-\varepsilon)$. Then, $\eqref{Kbarboundineq}$ implies
$$
-\frac{63}{134}(1-\varepsilon)S^2  + 8(1-\varepsilon) \geq \left(-\frac{1134}{67}(1-\varepsilon)^2+8\right)(1-\varepsilon) =0.
$$ Therefore, returning to (\ref{4_IV_final}), one deduces that 

\begin{equation}
     0 \geq \int_M  \Big(|\nabla W^+|^2 + |\nabla W^-|^2 \Big)dV_g.
\end{equation} Consequently, the Weyl tensor is harmonic. So, it follows from the works \cite{FLGR,MS,Wu} (see also \cite{CMM,yang2017rigidity}) that $(M^4,\,g)$ is Einstein. 

Finally, observing that in this case we have $K = \overline{K}>\varepsilon,$ with $\varepsilon>(\sqrt{1249}-23) / 120,$ it suffices to apply \cite[Theorem 1.1]{yang2000rigidity} to conclude that $(M^4,\,g)$ is iso\-me\-tric to either the standard sphere $\mathbb{S}^4$, or the complex projective space $\mathbb{C} \mathbb{P}^2.$ This completes the proof of the theorem. 
\end{proof}

We now proceed with the proof of Theorem \ref{thmNew}, which follows as a particular case of the following more general theorem.

\begin{theorem}
\label{Thm2c}
Let $(M^4,\,g,\,f)$ be an oriented $4$-dimensional compact gradient shrin\-king Ricci soliton satisfying \eqref{grs}. For any $t \in (0,1),$ consider $\varepsilon$ and $\delta$ such that 
\begin{equation}
\label{Thm2cA1}
     \left(\left(\frac{1-t}{2}\right)\delta + 2\delta - 12(1-\varepsilon)\right)>0   
     \end{equation}
    and
    \begin{equation}
\label{Thm2cA12}
    \frac{t\delta}{9} - \frac{4\delta}{3}+8(1-\varepsilon)+ 9\left(\frac{\delta}{3} - 2(1-\varepsilon) - \frac{t \delta}{54}\right)(1-\varepsilon)^2 >0.
\end{equation} Suppose that $\overline{K} \geq \varepsilon$, $S \geq \delta$ and 
\begin{equation}
\label{Thm2cA2}
    \int_{M}|\delta W^{+}|^2 dV_g \leq \int_{M} \frac{S}{6}|W^{+}|^2 dV_g.
\end{equation}
Then $M^4$ is either Einstein, or isometric to the standard sphere $\mathbb{S}^4,$ or  the complex projective space $\mathbb{C} \mathbb{P}^2.$ 
\end{theorem}

By varying $t \in (0,1)$, one obtains a non-trivial region in the
$(\delta,\,\varepsilon)$-plane where inequalities \eqref{Thm2cA1} and \eqref{Thm2cA12} are simultaneously satisfied. For $t=0.465$, we have the following green region.

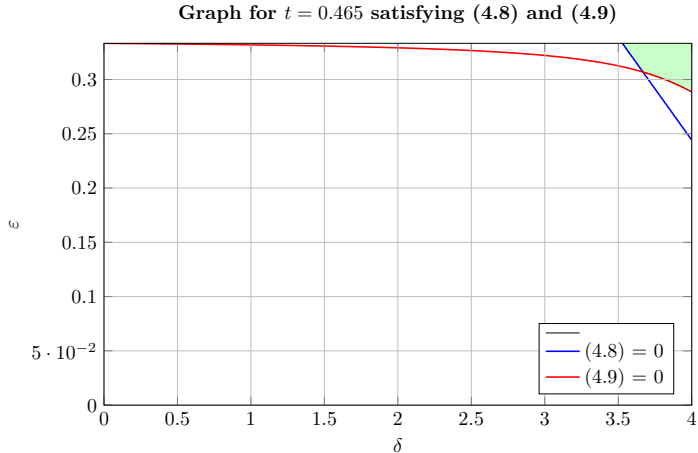
\vspace{0.30cm}

\begin{tikzpicture}[scale=0.75]
\begin{axis}[
    width=12cm, height=8cm,
    xlabel={$\delta$},
    ylabel={$\varepsilon$},
    title={{\bf Graph for $t=0.465$ satisfying (\ref{Thm2cA1}) and (\ref{Thm2cA12})}},
    xmin=0.0, xmax=4.0,
    ymin=0.0, ymax=0.3333333333333333,
    legend pos=south east,
    grid=both,
    enlargelimits=false,
    axis on top,
    clip=true
]
\addplot[fill=green!60, fill opacity=0.35, draw=none] coordinates {(3.996338,0.288815) (3.988643,0.289371) (3.980848,0.289928) (3.972951,0.290484) (3.964947,0.291041) (3.956834,0.291597) (3.948606,0.292154) (3.940259,0.292710) (3.931790,0.293267) (3.923193,0.293823) (3.914463,0.294380) (3.905596,0.294936) (3.896586,0.295492) (3.887427,0.296049) (3.878114,0.296605) (3.868639,0.297162) (3.858996,0.297718) (3.849178,0.298275) (3.839178,0.298831) (3.828987,0.299388) (3.818597,0.299944) (3.808000,0.300501) (3.797185,0.301057) (3.786143,0.301614) (3.774863,0.302170) (3.763334,0.302727) (3.751544,0.303283) (3.739480,0.303840) (3.727127,0.304396) (3.714472,0.304953) (3.701499,0.305509) (3.688191,0.306066) (3.674530,0.306622) (3.666530,0.307179) (3.663585,0.307735) (3.660640,0.308292) (3.657695,0.308848) (3.654750,0.309405) (3.651805,0.309961) (3.648860,0.310518) (3.645915,0.311074) (3.642970,0.311630) (3.640025,0.312187) (3.637080,0.312743) (3.634135,0.313300) (3.631190,0.313856) (3.628245,0.314413) (3.625300,0.314969) (3.622355,0.315526) (3.619410,0.316082) (3.616465,0.316639) (3.613520,0.317195) (3.610575,0.317752) (3.607630,0.318308) (3.604685,0.318865) (3.601740,0.319421) (3.598795,0.319978) (3.595850,0.320534) (3.592905,0.321091) (3.589960,0.321647) (3.587015,0.322204) (3.584070,0.322760) (3.581125,0.323317) (3.578180,0.323873) (3.575235,0.324430) (3.572290,0.324986) (3.569345,0.325543) (3.566400,0.326099) (3.563455,0.326656) (3.560510,0.327212) (3.557565,0.327769) (3.554620,0.328325) (3.551675,0.328881) (3.548730,0.329438) (3.545785,0.329994) (3.542840,0.330551) (3.539895,0.331107) (3.536950,0.331664) (3.534005,0.332220) (3.531060,0.332777) (3.528115,0.333333) (4.000000,0.333333) (4.000000,0.332777) (4.000000,0.332220) (4.000000,0.331664) (4.000000,0.331107) (4.000000,0.330551) (4.000000,0.329994) (4.000000,0.329438) (4.000000,0.328881) (4.000000,0.328325) (4.000000,0.327769) (4.000000,0.327212) (4.000000,0.326656) (4.000000,0.326099) (4.000000,0.325543) (4.000000,0.324986) (4.000000,0.324430) (4.000000,0.323873) (4.000000,0.323317) (4.000000,0.322760) (4.000000,0.322204) (4.000000,0.321647) (4.000000,0.321091) (4.000000,0.320534) (4.000000,0.319978) (4.000000,0.319421) (4.000000,0.318865) (4.000000,0.318308) (4.000000,0.317752) (4.000000,0.317195) (4.000000,0.316639) (4.000000,0.316082) (4.000000,0.315526) (4.000000,0.314969) (4.000000,0.314413) (4.000000,0.313856) (4.000000,0.313300) (4.000000,0.312743) (4.000000,0.312187) (4.000000,0.311630) (4.000000,0.311074) (4.000000,0.310518) (4.000000,0.309961) (4.000000,0.309405) (4.000000,0.308848) (4.000000,0.308292) (4.000000,0.307735) (4.000000,0.307179) (4.000000,0.306622) (4.000000,0.306066) (4.000000,0.305509) (4.000000,0.304953) (4.000000,0.304396) (4.000000,0.303840) (4.000000,0.303283) (4.000000,0.302727) (4.000000,0.302170) (4.000000,0.301614) (4.000000,0.301057) (4.000000,0.300501) (4.000000,0.299944) (4.000000,0.299388) (4.000000,0.298831) (4.000000,0.298275) (4.000000,0.297718) (4.000000,0.297162) (4.000000,0.296605) (4.000000,0.296049) (4.000000,0.295492) (4.000000,0.294936) (4.000000,0.294380) (4.000000,0.293823) (4.000000,0.293267) (4.000000,0.292710) (4.000000,0.292154) (4.000000,0.291597) (4.000000,0.291041) (4.000000,0.290484) (4.000000,0.289928) (4.000000,0.289371) (4.000000,0.288815) (3.996338,0.288815)};
\addlegendentry{}

\addplot[thick, color=blue] coordinates {(3.999315,0.244296) (3.996370,0.244853) (3.993425,0.245409) (3.990480,0.245965) (3.987535,0.246522) (3.984590,0.247078) (3.981645,0.247635) (3.978700,0.248191) (3.975755,0.248748) (3.972810,0.249304) (3.969865,0.249861) (3.966920,0.250417) (3.963975,0.250974) (3.961030,0.251530) (3.958085,0.252087) (3.955140,0.252643) (3.952195,0.253200) (3.949250,0.253756) (3.946305,0.254313) (3.943360,0.254869) (3.940415,0.255426) (3.937470,0.255982) (3.934525,0.256539) (3.931580,0.257095) (3.928635,0.257652) (3.925690,0.258208) (3.922745,0.258765) (3.919800,0.259321) (3.916855,0.259878) (3.913910,0.260434) (3.910965,0.260991) (3.908020,0.261547) (3.905075,0.262104) (3.902130,0.262660) (3.899185,0.263216) (3.896240,0.263773) (3.893295,0.264329) (3.890350,0.264886) (3.887405,0.265442) (3.884460,0.265999) (3.881515,0.266555) (3.878570,0.267112) (3.875625,0.267668) (3.872680,0.268225) (3.869735,0.268781) (3.866790,0.269338) (3.863845,0.269894) (3.860900,0.270451) (3.857955,0.271007) (3.855010,0.271564) (3.852065,0.272120) (3.849120,0.272677) (3.846175,0.273233) (3.843230,0.273790) (3.840285,0.274346) (3.837340,0.274903) (3.834395,0.275459) (3.831450,0.276016) (3.828505,0.276572) (3.825560,0.277129) (3.822615,0.277685) (3.819670,0.278242) (3.816725,0.278798) (3.813780,0.279354) (3.810835,0.279911) (3.807890,0.280467) (3.804945,0.281024) (3.802000,0.281580) (3.799055,0.282137) (3.796110,0.282693) (3.793165,0.283250) (3.790220,0.283806) (3.787275,0.284363) (3.784330,0.284919) (3.781385,0.285476) (3.778440,0.286032) (3.775495,0.286589) (3.772550,0.287145) (3.769605,0.287702) (3.766660,0.288258) (3.763715,0.288815) (3.760770,0.289371) (3.757825,0.289928) (3.754880,0.290484) (3.751935,0.291041) (3.748990,0.291597) (3.746045,0.292154) (3.743100,0.292710) (3.740155,0.293267) (3.737210,0.293823) (3.734265,0.294380) (3.731320,0.294936) (3.728375,0.295492) (3.725430,0.296049) (3.722485,0.296605) (3.719540,0.297162) (3.716595,0.297718) (3.713650,0.298275) (3.710705,0.298831) (3.707760,0.299388) (3.704815,0.299944) (3.701870,0.300501) (3.698925,0.301057) (3.695980,0.301614) (3.693035,0.302170) (3.690090,0.302727) (3.687145,0.303283) (3.684200,0.303840) (3.681255,0.304396) (3.678310,0.304953) (3.675365,0.305509) (3.672420,0.306066) (3.669475,0.306622) (3.666530,0.307179) (3.663585,0.307735) (3.660640,0.308292) (3.657695,0.308848) (3.654750,0.309405) (3.651805,0.309961) (3.648860,0.310518) (3.645915,0.311074) (3.642970,0.311630) (3.640025,0.312187) (3.637080,0.312743) (3.634135,0.313300) (3.631190,0.313856) (3.628245,0.314413) (3.625300,0.314969) (3.622355,0.315526) (3.619410,0.316082) (3.616465,0.316639) (3.613520,0.317195) (3.610575,0.317752) (3.607630,0.318308) (3.604685,0.318865) (3.601740,0.319421) (3.598795,0.319978) (3.595850,0.320534) (3.592905,0.321091) (3.589960,0.321647) (3.587015,0.322204) (3.584070,0.322760) (3.581125,0.323317) (3.578180,0.323873) (3.575235,0.324430) (3.572290,0.324986) (3.569345,0.325543) (3.566400,0.326099) (3.563455,0.326656) (3.560510,0.327212) (3.557565,0.327769) (3.554620,0.328325) (3.551675,0.328881) (3.548730,0.329438) (3.545785,0.329994) (3.542840,0.330551) (3.539895,0.331107) (3.536950,0.331664) (3.534005,0.332220) (3.531060,0.332777) (3.528115,0.333333)};
\addlegendentry{(\ref{Thm2cA1}) = 0}

\addplot[thick, color=red] coordinates {(3.996338,0.288815) (3.988643,0.289371) (3.980848,0.289928) (3.972951,0.290484) (3.964947,0.291041) (3.956834,0.291597) (3.948606,0.292154) (3.940259,0.292710) (3.931790,0.293267) (3.923193,0.293823) (3.914463,0.294380) (3.905596,0.294936) (3.896586,0.295492) (3.887427,0.296049) (3.878114,0.296605) (3.868639,0.297162) (3.858996,0.297718) (3.849178,0.298275) (3.839178,0.298831) (3.828987,0.299388) (3.818597,0.299944) (3.808000,0.300501) (3.797185,0.301057) (3.786143,0.301614) (3.774863,0.302170) (3.763334,0.302727) (3.751544,0.303283) (3.739480,0.303840) (3.727127,0.304396) (3.714472,0.304953) (3.701499,0.305509) (3.688191,0.306066) (3.674530,0.306622) (3.660497,0.307179) (3.646072,0.307735) (3.631231,0.308292) (3.615951,0.308848) (3.600206,0.309405) (3.583968,0.309961) (3.567208,0.310518) (3.549891,0.311074) (3.531984,0.311630) (3.513447,0.312187) (3.494238,0.312743) (3.474311,0.313300) (3.453617,0.313856) (3.432100,0.314413) (3.409700,0.314969) (3.386351,0.315526) (3.361980,0.316082) (3.336507,0.316639) (3.309840,0.317195) (3.281882,0.317752) (3.252521,0.318308) (3.221634,0.318865) (3.189081,0.319421) (3.154707,0.319978) (3.118334,0.320534) (3.079763,0.321091) (3.038766,0.321647) (2.995082,0.322204) (2.948414,0.322760) (2.898416,0.323317) (2.844688,0.323873) (2.786764,0.324430) (2.724095,0.324986) (2.656032,0.325543) (2.581806,0.326099) (2.500489,0.326656) (2.410961,0.327212) (2.311849,0.327769) (2.201461,0.328325) (2.077677,0.328881) (1.937812,0.329438) (1.778414,0.329994) (1.594964,0.330551) (1.381432,0.331107) (1.129592,0.331664) (0.827915,0.332220) (0.459730,0.332777) (0.000000,0.333333)};
\addlegendentry{(\ref{Thm2cA12}) = 0}

\addplot[thin, color=black, dotted] coordinates {(3.996338,0.288815) (3.988643,0.289371) (3.980848,0.289928) (3.972951,0.290484) (3.964947,0.291041) (3.956834,0.291597) (3.948606,0.292154) (3.940259,0.292710) (3.931790,0.293267) (3.923193,0.293823) (3.914463,0.294380) (3.905596,0.294936) (3.896586,0.295492) (3.887427,0.296049) (3.878114,0.296605) (3.868639,0.297162) (3.858996,0.297718) (3.849178,0.298275) (3.839178,0.298831) (3.828987,0.299388) (3.818597,0.299944) (3.808000,0.300501) (3.797185,0.301057) (3.786143,0.301614) (3.774863,0.302170) (3.763334,0.302727) (3.751544,0.303283) (3.739480,0.303840) (3.727127,0.304396) (3.714472,0.304953) (3.701499,0.305509) (3.688191,0.306066) (3.674530,0.306622) (3.666530,0.307179) (3.663585,0.307735) (3.660640,0.308292) (3.657695,0.308848) (3.654750,0.309405) (3.651805,0.309961) (3.648860,0.310518) (3.645915,0.311074) (3.642970,0.311630) (3.640025,0.312187) (3.637080,0.312743) (3.634135,0.313300) (3.631190,0.313856) (3.628245,0.314413) (3.625300,0.314969) (3.622355,0.315526) (3.619410,0.316082) (3.616465,0.316639) (3.613520,0.317195) (3.610575,0.317752) (3.607630,0.318308) (3.604685,0.318865) (3.601740,0.319421) (3.598795,0.319978) (3.595850,0.320534) (3.592905,0.321091) (3.589960,0.321647) (3.587015,0.322204) (3.584070,0.322760) (3.581125,0.323317) (3.578180,0.323873) (3.575235,0.324430) (3.572290,0.324986) (3.569345,0.325543) (3.566400,0.326099) (3.563455,0.326656) (3.560510,0.327212) (3.557565,0.327769) (3.554620,0.328325) (3.551675,0.328881) (3.548730,0.329438) (3.545785,0.329994) (3.542840,0.330551) (3.539895,0.331107) (3.536950,0.331664) (3.534005,0.332220) (3.531060,0.332777) (3.528115,0.333333)};

\end{axis}

\end{tikzpicture}

\vspace{0.30cm}

\begin{proof}
To begin with, notice from (\ref{4_Weizenbock_unormalised_int}) that

\begin{equation*}
    0 \geq \int_M \Big( |\nabla W^+|^2+ \frac{S}{2}|W^+|^2 - \sqrt{6}|W^+|^3 - \frac{1}{\sqrt{6}}|\mathring{Ric}|^2|W^+| \Big)dV_g.
\end{equation*} Then, we use the last inequality in \eqref{Kbarboundineq} to infer 
\begin{eqnarray*}
       0 &\geq &  \int_M \Big[|\nabla W^+|^2 + \frac{S}{2}|W^+|^2+\left(2S-12(1-\varepsilon)\right)|W^+|^2 \nonumber\\&&+\left(\frac{S}{3}-2(1-\varepsilon)\right)|\mathring{Ric}|^2 \Big] dV_g.
\end{eqnarray*} Since $S \geq \delta,$ one sees that

\begin{eqnarray}
\label{123er41}
        0 & \geq & \int_M \Big[ |\nabla W^+|^2  + \frac{\delta}{2}|W^+|^2+\left(2\delta-12(1-\varepsilon)\right)|W^+|^2 \nonumber\\&&+\left(\frac{\delta}{3}-2(1-\varepsilon)\right)|\TRic|^2\Big]\,dV_g\nonumber\\
       & = &\int_M \Big[ |\nabla W^+|^2  + \left(\left(\frac{1-t}{2}\right)\delta + 2\delta - 12(1-\varepsilon)\right)|W^+|^2 + \frac{t\delta}{2} |W^+|^2\nonumber\\
       &&  +\left(\frac{\delta}{3}-2(1-\varepsilon)\right)|\TRic|^2 \Big] \,dV_g,
\end{eqnarray} where $t\in(0,1)$ will be determined later.

At this time, we proceed by assuming that

\begin{equation}
\label{4_gap_assumption_new}
    \int_M |W^+|^2dV_g \geq \frac{1}{9} \int_M \left( \frac{S^2}{24}-\frac{|\TRic|^2}{2}\right)dV_g.
\end{equation} In this situation, we may use (\ref{intricball}) to infer

\begin{equation}
    \int_M|W^+|^2 dV_g \geq \frac{1}{9} \int_M \left(2-\frac{S^2}{12}\right)dV_g.
\end{equation} This substituted into (\ref{123er41}) yields

\begin{eqnarray*}
    0 & \geq & \int_M\Big[ |\nabla W^+|^2  + \left(\left(\frac{1-t}{2}\right)\delta + 2\delta - 12(1-\varepsilon)\right)|W^+|^2  \\
       && + \frac{t\delta}{9} - \frac{4\delta}{3}+8(1-\varepsilon)+ \left(\frac{\delta}{3} - 2(1-\varepsilon) - \frac{t \delta}{54}\right)\frac{S^2}{4}\Big]dV_g. 
\end{eqnarray*} Taking into account that $S^2 \leq 36(1-\varepsilon)^2,$ and noting that the last term on the right-hand side is nonpositive, we obtain 
    \begin{eqnarray}
    \label{4thmc_choice_delta_ep_t}
         0 &\geq& \int_M \Big[|\nabla W^+|^2  + \left(\left(\frac{1-t}{2}\right)\delta + 2\delta - 12(1-\varepsilon)\right)|W^+|^2\nonumber \\
       && + \frac{t\delta}{9} - \frac{4\delta}{3}+8(1-\varepsilon)+ 9\left(\frac{\delta}{3} - 2(1-\varepsilon) - \frac{t \delta}{54}\right)(1-\varepsilon)^2\Big]dV_g.
    \end{eqnarray} Notice that, under conditions \eqref{Thm2cA1} and \eqref{Thm2cA12}, the second line on the right-hand side of \eqref{4thmc_choice_delta_ep_t} is nonnegative. Consequently,  \eqref{4thmc_choice_delta_ep_t} yields 
$$\int_M |\nabla W^+|^2\,dV_g\leq 0,$$ which implies that $\nabla W^+=0.$ Hence, $\delta W^+=0,$ that is, the manifold has half-harmonic Weyl curvature. It then suffices to apply \cite[Theorem 1.1]{Wu} to conclude that $M^4$ is Einstein. Furthermore, observing that the sectional curvature satisfies $K = \overline{K}>\varepsilon$ with $\varepsilon>(\sqrt{1249}-23) / 120,$ which is the sharp lower bound established in \cite[Theorem 1.1]{yang2000rigidity}, we invoke the rigidity result therein to conclude that $(M^4,\,g)$ is isometric to either $\mathbb{S}^4$ or the complex projective space $\mathbb{CP}^2.$

To conclude, we need to show why \eqref{4_gap_assumption_new} must hold. Suppose, by contradiction, that
\begin{equation}
\label{eqklj90p}
    \int_M |W^+|^2\,dV_g <\frac{1}{9} \int_M \left( \frac{S^2}{24}-\frac{|\TRic|^2}{2}\right)\,dV_g.
\end{equation} Note that both sides of this inequality are conformally invariant. By \cite[Proposition 4.5 and Remark 4.1]{CH}, assumption \eqref{Thm2cA2} implies that the weighted Yamabe functional satisfies $\widehat{Y}_{1,6\sqrt{6}}(M)\leq 0.$ Then, by \cite[Proposition 4.4]{CH}, there exists a smooth conformal metric $\widetilde{g}=u^2g$ such that 
\begin{equation}
\label{Prop4.5}
\int_M |\widetilde{W}^{+}|^2 d V_{\widetilde{g}} \geq  \frac{1}{216}\int_M \widetilde{\mathrm{~S}}^2 d V_{\widetilde{g}}.
\end{equation} However, applying the same conformal scaling to  (\ref{eqklj90p}), one obtains that 

\begin{equation}
\begin{aligned}
     \int_M |\widetilde{W}^+|^2dV_{\widetilde{g}} &< \frac{1}{9}\int_M \left(\frac{\widetilde{S}^2}{24} - \frac{|\overset{\circ}{\widetilde{\text{Ric}}}|^2}{2}\right)dV_{\widetilde{g}}\\
     &< \frac{1}{216}\int_M\widetilde{S}^2dV_{\widetilde{g}},
\end{aligned}
\end{equation} which contradicts \eqref{Prop4.5}. So, the proof is completed. 

\end{proof}

\begin{proof}[{\bf Proof of Theorem \ref{thmNew}}]
       Let $t=0.465$, $\delta \approx 3.668$ and $\varepsilon \approx 0.3069$. With these values, \eqref{Thm2cA1} and \eqref{Thm2cA12} are satisfied, and thus Theorem~\ref{Thm2c} applies.
\end{proof}

\vspace{0.30cm}

\subsection{Weighted integral bound}
Here, we present the proof of Theorem \ref{ThmC}. To this end, it is important to recall the following weighted Lichnerowicz-type estimate (cf. \cite{BakryEmery,WW}).

\begin{lemma}[\cite{BakryEmery,WW}]
\label{lemWW}
Let $(M^n,\,g)$ be a complete Riemannian manifold and $f:M\to \mathbb{R}.$  If $Ric_f \geq (n-1) \sigma,$ then $$\lambda_1 \geq (n-1)\sigma,$$ where $\lambda_1$ is the first eigenvalue of the weighted Laplacian operator. 
\end{lemma}

For the reader's convenience, we restate Theorem \ref{ThmC} below.

\begin{theorem}[Theorem \ref{ThmC}]
    Let $(M^4,\,g,\,f)$ be a $4$-dimensional compact gradient shrinking Ricci soliton satisfying \eqref{grs}. Suppose that $\overline{K} \geq \varepsilon $ and $S \geq \delta$ such that $\frac{21}{2}<2 \delta+12 \varepsilon$. Then, for either $|W^+|$ or $|W^-|$, the following inequality holds:
    \begin{equation}\label{4_III_Gap}
        \int_M|W^\pm|^2 e^{-f}\,dV_g \leq \alpha\int_M S^2 e^{-f}\,dV_g,
    \end{equation} where $\alpha =\frac{4(1-\varepsilon)}{\delta}-\frac{2}{3}-\left(1-\varepsilon\right)+ \frac{\delta}{6}$.
\end{theorem}

\begin{proof}
Initially, assume by contradiction that both $|W^+|$ and $|W^-|$ satisfy 
    \begin{equation}\label{4_T1_Gap}
        \int_M|W^\pm|^2e^{-f}dV_g > \alpha\int_MS^2e^{-f}dV_g. 
    \end{equation}

At the same time, it follows from (\ref{3_lapW+}) that 
 \begin{equation}\label{wwww12}
            \frac{1}{2}\Delta_f\left|\mathrm{W}^{+}\right|^2 \geq  \left|\nabla \mathrm{W}^{+}\right|^2 +2\left|\mathrm{W}^{+}\right|^2-\sqrt{6}| \mathrm{W}^{+}|^3-\frac{1}{\sqrt{6}}|\mathring{Ric}|^2|W^+|.
        \end{equation} Moreover, by using \eqref{Kbarboundineq}, one sees that     $$
        \frac{1}{\sqrt{6}}(|W^+|+|W^-|) \leq 2(1-\varepsilon) - \frac{S}{3} \leq  2(1-\varepsilon) - \frac{\delta}{3}.
    $$ Plugging this into (\ref{wwww12}) yields

    \begin{eqnarray}
    \label{klp014gh}
\frac{1}{2} \Delta_f |W^+|^2 &\geq & |\nabla W^+|^2 +2|W^+|^2-12(1-\varepsilon)|W^+|^2+2\delta|W^+|^2\nonumber\\&&-\left(2(1-\varepsilon)-\frac{\delta}{3}\right)|\mathring{Ric}|^2.
\end{eqnarray} 

On the other hand, on integrating (\ref{trRic_RiclapS}), we get 

\begin{equation}\label{trless-ricci-int}
    \begin{aligned}
    \int_M |\mathring{Ric}|^2e^{-f}dV_g =  \int_M Se^{-f}dV_g - \int_M\frac{1}{4}S^2e^{-f}dV_g.
\end{aligned}
\end{equation} Hence, upon integrating (\ref{klp014gh}) with the weighted volume, we use (\ref{trless-ricci-int}) to infer

\begin{eqnarray}
\label{4_pre_Yang1}
    0 \geq \int_M\left( |\nabla W^+|^2 + (2\delta+12\varepsilon -10)|W^+|^2 - \left(2(1-\varepsilon)-\frac{\delta}{3}\right)\left(S-\frac{S^2}{4}\right)\right)e^{-f} dV_g. \nonumber\\
\end{eqnarray} Similarly, we obtain inequality (\ref{4_pre_Yang1}) for $W^-.$ From this point on, we follow the approach outlined in \cite{yang2000rigidity}. Specifically, by assuming that $(M^4,\,g)$ is not half-conformally flat, we have 
\begin{equation}
\int_M\left|W^{+}\right|e^{-f} dV_g>0\,\,\, \text { and }\,\,\,\int_M\left|W^{-}\right| e^{-f}dV_g>0.
\end{equation} Thus, there exists a positive constant $t>0$ such that 
$$\int_M\left(\left|W^{+}\right|-t\left|W^{-}\right|\right) e^{-f}\,dV_g=0.
$$ 

Next, the Kato inequality guarantees that

$$\left|\nabla W^{+}\right|^2+t^2\left|\nabla W^{-}\right|^2 \geq |\nabla| W^{+}| |^2+t^2 | \nabla |W^{-}||^2.$$ Moreover, note that

\begin{eqnarray*}
    \left(|\nabla |W^+||^2 + t^2|\nabla |W^- ||^2 \right) &=& \frac{1}{2} \Big(|\nabla \left(|W^+|-t|W^- |\right)|^2 + |\nabla \left(|W^+| + t |W^- |\right)|^2 \Big)\nonumber\\ &\geq & \frac{1}{2}|\nabla \left( |W^+| - t |W^-|\right)|^2. 
\end{eqnarray*} Combining these previous expressions, one obtains that

\begin{eqnarray*}
    \int_M \left(\left|\nabla W^+\right|^2+t^2\left|\nabla W^{-}\right|^2 \right) e^{-f}dV_g & \geq & \frac{1}{2}\int_M \left|\nabla\left(\left|W^{+}\right|-t | W^{-}|\right)\right|^2 e^{-f}dV_g \nonumber\\
    & \geq &  \frac{\lambda_1}{2}\int_M \left|\left(\left|W^{+}\right|-t | W^{-}|\right)\right|^2 e^{-f}dV_g.\nonumber
\end{eqnarray*} Now, for a solution to \eqref{grs} in dimension four, we use  Lemma \ref{lemWW} with $\sigma =1/3$ in order to obtain $\lambda_1 \geq 1$ (not optimal for compact manifolds, see \cite{WW}). In this situation, one obtains that

\begin{equation}\label{Yang_gradient}
    \int_M \left(\left|\nabla W^+\right|^2+t^2\left|\nabla W^{-}\right|^2 \right) e^{-f}\,dV_g  \geq \frac{1}{2}\int_M \left|\left(\left|W^{+}\right|-t | W^{-}|\right)\right|^2 e^{-f}\,dV_g.
\end{equation} Substituting \eqref{4_pre_Yang1} into \eqref{Yang_gradient} gives

\begin{eqnarray*}
    0 & \geq & \int_M \left[\left( \frac{1}{2}+ (2\delta+12\varepsilon -10)\right)|W^+|^2-\left(2(1-\varepsilon)-\frac{\delta}{3}\right)\left(S-\frac{S^2}{4}\right)\right]e^{-f}dV_g \nonumber \\
    & - &t \int_M |W^+||W^-|e^{-f}dV_g\nonumber \\
    & + & t^2  \int_M \left[\left( \frac{1}{2}+ (2\delta+12\varepsilon -10)\right)|W^-|^2-\left(2(1-\varepsilon)-\frac{\delta}{3}\right)\left(S-\frac{S^2}{4}\right)\right]e^{-f} dV_g,
\end{eqnarray*} and by \eqref{4_T1_Gap}, one deduces that 

\begin{eqnarray}
\label{III_yang_contradiction}
    0 & \geq & \int_M \left[(2\delta+12\varepsilon -10)|W^+|^2+ \phi(S)\right]e^{-f}dV_g \nonumber \\
    && -t\int_M |W^+||W^-|e^{-f}dV_g\nonumber \\
    && + t^2  \int_M \left[(2\delta+12\varepsilon -10)|W^-|^2+\phi(S)\right]e^{-f}dV_g,
\end{eqnarray} where 
\begin{equation*}
    \phi(S) := \left( \frac{\alpha}{2}+ \frac{(1-\varepsilon)}{2}-\frac{\delta}{12}\right)S^2 -\left(2(1-\varepsilon)-\frac{\delta}{3}\right)S.
\end{equation*}

Observe that the right-hand side of \eqref{III_yang_contradiction} defines a quadratic form $Q(t)$ in $t.$ In particular, its discriminant is given by

\begin{equation}\label{def_discriminant}
\begin{aligned}
        \Psi = |W^+|^2|W^-|^2 & - 4(2\delta+12\varepsilon -10)^2|W^+|^2|W^-|^2  \\
        & - 4(2\delta+12\varepsilon -10)|W^-|^2\phi(S) \\
        & - 4(2\delta+12\varepsilon -10)|W^+|^2\phi(S) \\
        & - 4\phi(S)^2.
\end{aligned}
\end{equation} Hence, the quadratic form \eqref{III_yang_contradiction} is positive definite whenever $$1-4(2\delta+12\varepsilon -10)^2<0,$$ which holds if $\frac{21}{2}<2 \delta+12 \varepsilon.$ Moreover, when $t=0$, we observe that $Q(0)>0$, since the coefficient of $|W^+|^2$ satisfies $2\delta+12 \varepsilon-10>\frac{1}{2}$ and $\phi(S)>0$ for $S>\delta.$  Therefore, the right-hand side of \eqref{III_yang_contradiction} is positive for all $t>0$, leading to a contradiction. Consequently, the reverse inequality in \eqref{4_T1_Gap} must hold. So, the proof is completed.

\end{proof}

\subsection{Hitchin-Thorpe type inequality} Now, we shall present the proof of Theo\-rem \ref{Thm_HT}, which we restate below for the reader's convenience.

\begin{theorem}[Theorem \ref{Thm_HT}]
     Let $(M^4,\,g,\,f)$ be a $4$-dimensional oriented compact gradient shrinking Ricci soliton satisfying \eqref{grs} with $\overline{K} \geq \varepsilon = 0.184.$ Suppose that 
    \begin{equation}\label{6_Asump}
    \int_M |\delta W^{+}|^2\,dV_{g} \leq \int_M \frac{S}{6} |W^{+}|^2\, dV_{g}.
    \end{equation} Then the following inequality holds: 
    \begin{equation}
    \chi(M)>\frac{3}{2}\,|\tau(M)|.
    \end{equation}
\end{theorem}

\begin{proof}
To begin with, we use (\ref{intricball}) to infer

\begin{equation}\label{6_HTest1}
    \int_M |Ric|^2 dV_g = \int_M \frac{S^2}{2}dV_g - 4 Vol(M).
\end{equation}  Substituting \eqref{6_HTest1} and \eqref{Tricdef} into \eqref{6_HTest3} yields
\begin{equation}\label{6_HTest5}
    8\pi^2 \chi(M) =\int_M \left(|W^{+}|^2+|W^{-}|^2\right) dV_g+2 Vol(M)-\frac{1}{12} \int_{M} S^2 dV_g.
\end{equation} 

On the other hand, it follows from (\ref{Kbarboundineq}) that a lower bound on modified sectional curvature $\overline{K} \geq \varepsilon$ yields 
\begin{equation}\label{6chi_pre}
    |W^+|^2+|W^-|^2 \leq 24(1-2 \varepsilon+\varepsilon^2) + \frac{2}{3}S^2 - 8(1-\varepsilon)S.
\end{equation} Upon integrating \eqref{6chi_pre}, we use (\ref{6_HTest5}) in order to obtain
\begin{equation}
8 \pi^2 \chi(M)<\frac{7}{12} \int_M \mathrm{~S}^2 dV_g+2 \left(12 \varepsilon^2-8 \varepsilon-3\right) Vol(M),
\end{equation} where we also used that $$\int_M S \,dV_g = 4Vol(M),$$ which follows from (\ref{traceGRS}). Besides, the scalar curvature upper bound, resulting from the lower bound on the modified sectional curvature, implies that 
\begin{equation}\label{6_HTest6}
8 \pi^2 \chi(M)<\underbrace{(21(1-\varepsilon)^2+2\left(12 \varepsilon^2-8 \varepsilon-3\right))}_{\psi(\varepsilon)}Vol(M).
\end{equation}

Returning now to \eqref{6_HTest5}, we use (\ref{6_HTest311}) to obtain 
\begin{eqnarray}
\label{kjl150okp}
2 \chi(M)+3 \tau(M) &= &  \frac{1}{2 \pi^2} \int_M  |W^{+}|^2 dV_g+\frac{1}{2 \pi^2} Vol(M)\nonumber\\&&-\frac{1}{48 \pi^2} \int_M S^2 dV_g.
\end{eqnarray} Next, by using \eqref{6_Asump}, we apply  \cite[Theorem 4.1 and Remark 4.1]{CH} to deduce

\begin{equation}\label{6_HalfWeylLower}
\int_M|W^+|^2 dV_g> \frac{4}{11}\pi^2 (2\chi(M) + 3\tau(M)). 
\end{equation} Plugging this into (\ref{kjl150okp}), one sees that  
$$
\begin{aligned}
        2\chi(M) + 3 \tau(M) =& \frac{1}{2\pi^2}\int_M|W^+|^2dV_g + \frac{1}{2\pi^2}Vol(M) -\frac{1} {48\pi^2}\int_MS^2dV_g \\
        & \geq \frac{2}{11}(2\chi(M)+3 \tau(M)) + \frac{1}{2\pi^2}Vol(M) - \frac{3}{4 \pi^2} (1-\varepsilon)^2 Vol(M) \\
        &= \frac{2}{11}(2\chi(M)+3 \tau(M)) + \underbrace{\left( \frac{1}{2\pi^2} - \frac{3}{4 \pi^2}(1-\varepsilon)^2\right)}_{\gamma(\varepsilon)}Vol(M). 
\end{aligned}
$$
From now on, we restrict our choice of $\varepsilon< 1/3$ so that $\gamma(\varepsilon)>0$ and $\psi(\varepsilon)>0$. By using \eqref{6_HTest6}, one obtains that  
\begin{equation}\label{6_HTineq1}
    \left(\frac{18}{27}- \frac{\gamma(\varepsilon)}{\psi(\varepsilon)}\frac{88 \pi^2}{27}\right)\chi(M) > - \tau(M).
\end{equation}

On the other hand, we reverse the orientation so that  \eqref{6_HalfWeylLower} applies to $|W^-|$. Proceeding as in the previous case, we then obtain
\begin{equation}
    2 \chi(M) - 3 \tau(M) = \frac{1}{2 \pi^2} \int_M |W^-|^2dV_g - \frac{1}{48 \pi^2}\int_MS^2dV_g + \frac{1}{2 \pi^2} Vol(M).
\end{equation} Besides, by using \eqref{6_HalfWeylLower}, \eqref{Kbarboundineq} and \eqref{6_HTest6}, we get 

\begin{equation}\label{6_HTineq2}
    \left(\frac{18}{39} - \frac{88 \pi^2}{39} \frac{\gamma(\varepsilon)}{\psi(\varepsilon)} \right) \chi(M)> \tau(M). 
\end{equation} Combining \eqref{6_HTineq1}, \eqref{6_HTineq2} and let $\varepsilon = 0.184$, we then obtain 
\begin{equation*}
    -0.6663\chi(M)<\tau(M) < 0.4613\chi(M),
\end{equation*} In particular, we have 
\begin{equation*}
    |\tau(M)| \leq 0.6663 \chi(M),
\end{equation*}
and hence
\[
\chi(M)>
\frac{1}{0.6663}\,|\tau(M)|
>\frac{3}{2}\,|\tau(M)|.
\]
This finishes the proof of the theorem. 
\end{proof}

\end{document}